\documentclass[12pt,twoside,a4paper]{amsart}
\usepackage{stmaryrd}
\usepackage[T1]{fontenc}
\usepackage[latin1]{inputenc}
\usepackage[left=30mm,top=30mm,bottom=30mm]{geometry}
\usepackage{amssymb}
\usepackage{mathtools}
\usepackage{lmodern}
\usepackage[
final,
]{prelim2e}

\renewcommand{\PrelimText}{\footnotesize[\,Version: 
\texttt{\jobname.tex}\hfill \today\ at \thistime\,]}
      \theoremstyle{plain}
      \newtheorem{Thm}{Theorem}[section]
      \newtheorem{Lem}[Thm]{Lemma}
      \newtheorem{Cor}[Thm]{Corollary}
      \newtheorem{Example}[Thm]{Example}
      \newtheorem{Subexample}[Thm]{Subexample}
      \newtheorem{Subsubexample}[Thm]{Subsubexample}
      \newtheorem{Counterexample}[Thm]{Counterexample}
      \theoremstyle{definition}
      \newtheorem{Def}[Thm]{Definition}
      \theoremstyle{remark}
      
\newcounter{aequiv}

\newenvironment{Aequiv}
  {\begin{list}{\rm(\roman{aequiv})}%
   {\usecounter{aequiv}%
     \setlength{\leftmargin}{0pt}%
     \setlength{\labelwidth}{2.8em}%
     \setlength{\itemindent}{1.4\labelwidth}%
     \setlength{\partopsep}{0pt}%
     \setlength{\topsep}{0pt}%
     \setlength{\labelsep}{1.3em}%
     \setlength{\itemsep}{0.0ex}\setlength{\parsep}{0.0ex}%
   }}%
 {\end{list}} 
\let\oldsqrt\sqrt
\def\sqrt{\mathpalette\DHLhksqrt}
\def\DHLhksqrt#1#2{%
\setbox0=\hbox{$#1\oldsqrt{#2\,}$}\dimen0=\ht0
\advance\dimen0-0.2\ht0
\setbox2=\hbox{\vrule height\ht0 depth -\dimen0}%
{\box0\lower0.4pt\box2}}
\newcommand{\E}{{\mathbb E}}     %
\newcommand{\N}{{\mathbb N}} \renewcommand{\P}{{\mathbb P}}    %
\newcommand{\R}{{\mathbb R}}     %
\newcommand{\cA}{{\mathcal A}}\newcommand{\cB}{{\mathcal B}}
\newcommand{\cC}{{\mathcal C}}\newcommand{\cD}{{\mathcal D}}
\newcommand{\cE}{{\mathcal E}}\newcommand{\cF}{{\mathcal F}}

\newcommand{\cI}{{\mathcal I}}
\newcommand{\cL}{{\mathcal L}}

\newcommand{\cO}{{\mathcal O}}\newcommand{\cP}{{\mathcal P}}
\newcommand{\cQ}{{\mathcal Q}}\newcommand{\cR}{{\mathcal R}}

\newcommand{\cX}{{\mathcal X}}
\newcommand{\cY}{{\mathcal Y}}
\renewcommand{\epsilon}{\varepsilon}\renewcommand{\phi}{\varphi} %
\renewcommand{\rho}{\varrho}\renewcommand{\theta}{\vartheta}     %
\newcommand{\Alpha}{\mathrm{A}}    %
\newcommand{\Eta}{\mathrm{H}}      %
\newcommand{\defn}[1]{\emph{#1}}   
\renewcommand{\[}{\begin{eqnarray*}}\renewcommand{\]}{\end{eqnarray*}}
\newcommand{\la}{\begin{eqnarray}}\newcommand{\al}{\end{eqnarray}}
\newcommand{\id}{{\mbox{\rm id}}}  
\newcommand{\Prob}{\mbox{\rm Prob}}
\newcommand{\dd}{{\,\mathrm{d}}}
\newcommand{\1}{\mathbf 1}       
\newcommand{\lpp}{\llparenthesis}%
\newcommand{\rpp}{\rrparenthesis}%
   
\begin{document} 
\author[Kagan]{Abram M.\ Kagan} 
\address{Department of Mathematics, University of Maryland, College Park,
MD 20742, USA}
\email{amk@math.umd.edu}
\author[Malinovsky]{Yaakov Malinovsky}  
\address{Department of Mathematics and Statistics, 
University of Maryland Baltimore County, Baltimore,
MD 2150, USA}
\email{yaakovm@umbc.edu}
\author[Mattner]{Lutz Mattner}    
\address{Universit\"at Trier, Fachbereich IV -- Mathematik, 54286~Trier, 
Germany}
\email{mattner@uni-trier.de}



\title[Partially complete sufficient statistics are jointly complete]
{Partially complete sufficient statistics\\ are jointly  complete}


\begin{abstract} 
The theory of the basic statistical concept of 
(Lehmann-Scheff\'e-)com\-plete\-ness is perfected by providing 
the theorem indicated in the title and previously overlooked 
for several decades. Relations to earlier results are discussed and
illustrating examples are presented.

Of the two proofs offered for the main result, the first is direct 
and short, following the prototypical example of Landers and Rogge~(1976),
and the second is very short and purely statistical, 
utilizing the basic theory of optimal
unbiased estimation in the little known version completed by
Schmetterer and Strasser~(1974).
\end{abstract}



\thanks{The work of the second author was partially supported by a 2013 UMBC
Summer Faculty Fellowship grant.}


\date{\today}


\maketitle
\tableofcontents
\section{Introduction, main results, and discussion}\label{Sec:1}
The main purpose of this paper is to provide with Theorem~\ref{Thm:main}
below a result yielding (Lehmann-Scheff\'e-)completeness in possibly 
complicated statistical models as a consequence of completeness
in suitable submodels, and to illustrate the use
of this theorem with short proofs of some classical results. 
The latter include complete sufficiency in models involving truncation,
see Examples~\ref{Example:Intervals_vs_rays} 
and~\ref{Example:Compl_suff_unkknow_trunc}, 
Subexample~\ref{Subexample:Trunc_exp_fam}, 
and Subsubexample~\ref{Subsubex:LS_exponential}, and we present short and 
natural proofs of the requisite auxiliary results 
\ref{Lem:uniform_cap-stable_compl_suff},
\ref{Cor:Taxi_examples}, 
and~\ref{Lem:Compl_suff_weight}.
Example~\ref{Example:Compl_suff_unkknow_trunc} might be new 
in its present natural generality. 

The conclusion of joint sufficiency, absent from Theorem~\ref{Thm:main},
can be added under a homogeneity assumption, see Theorem~\ref{Thm:Main_hom_con} 
and Counterexample~\ref{CounterEx:C1_vee_C2_insuff}. 

A secondary purpose of this paper is  to correct or refute related 
completeness claims from the literature, see below Theorem~\ref{Thm:CKS} and 
the two paragraphs following it.

We present two proofs of  Theorem~\ref{Thm:main}, namely 
in Section~\ref{Sec:2} a short and direct one, 
generalizing the original proof of the prototypical
Example~\ref{Example:Landers_Rogge} given by  Landers and Rogge~(1976),
and, at the end of Section~\ref{Sec:3}, a very short 
and purely statistical one, utilizing 
the theory of optimal unbiased estimation as completed by 
Schmetterer and Strasser~(1974). 
The perhaps surprising possibility of the second proof rests on the 
apparently not too well-known fact that optimality of an unbiased estimator,
in the sense of Definition~\ref{Def:Opt_unb},
is always equivalent to its measurability with respect to a certain complete
but not necessarily sufficient sub-$\sigma$-algebra, 
namely the $\sigma$-algebra $\cO$ associated to the model $\cP$
in the known Theorem~\ref{Thm_RPLSSS}. Theorem~\ref{Thm_RPLSSS} 
and the trivial but useful and known Lemma~\ref{Lem:Bo_1983}  easily 
yield with Theorem~\ref{Thm:Partial_opt_implies_opt} 
a lower bound for $\cO$ in terms of the $\sigma$-algebras  
$\cO_\eta$ corresponding to submodels $\cP_\eta$ forming an exhaustion of $\cP$;
and Theorem~\ref{Thm:Partial_opt_implies_opt}  in turn allows a very short and 
computation free second proof of Theorem~\ref{Thm:main}.

To be more precise, let us introduce some notation and recall  
basic definitions. 
With $\Prob(\cX,\cA)$ denoting the set of all laws on the measurable 
space $(\cX,\cA)$, every set $\cP\subseteq \Prob(\cX,\cA)$ 
is a \defn{(statistical) model} on $(\cX,\cA)$, and then
every $\cQ\subseteq \cP$ is a \defn{submodel}, and every 
family $(\cP_\eta:\eta \in\Eta)$ of submodels $\cP_\eta$ of $\cP$
with $\bigcup_{\eta\in\Eta}\cP_\eta=\cP$
will here be called a \defn{(parametrized) exhaustion} of $\cP$.
The most common example of the latter is given for a model
$\cP=\{P_\theta:\theta\in\Theta\}$ with $\Theta=\Theta_1\times\Theta_2$
by $\cP_\eta:=\{P_{\theta_1,\eta}:\theta_1\in\Theta_1\}$ for 
$\eta\in\Theta_2$, as occurring in Corollary~\ref{Cor:I={1,2}} and there
in particular in assumption~(i), but other cases as in 
Example~\ref{Example:Intervals_vs_rays} are not uncommon.

Let $\cP\subseteq\Prob(\cX,\cA)$ be a model. 
A sub-$\sigma$-algebra $\cC$ of $\cA$ is \defn{complete} for $\cP$
if every $\cC$-measurable function $h:\cX\rightarrow \R$ with 
vanishing expectations under $\cP$, that is,
\la                \label{Eq:Hypothesis_of_completeness}
  Ph&=&0   \quad\text{ for }P\in\cP,
\al
already satisfies
\la                \label{Eq:Conclusion_of_completeness}
  h&=&0   \quad\text{$\cP$-a.s.},
\al
that is, $h=0$ $P$-a.s.~for every $P\in\cP$. The model~$\cP$ itself is 
\defn{complete} if $\cA$ is complete for~$\cP$. A statistic 
$S$ from $(\cX,\cA)$ to some measurable space $(\cY,\cB)$ is 
\defn{complete} for $\cP$
if the $\sigma$-algebra $\sigma(S)$ it generates on $\cX$ is complete for $\cP$. 
While logically unnecessary, statistics are common and often very convenient
for describing sub-$\sigma$-algebras in concrete examples,
such as Subsubexample~\ref{Subsubex:LS_exponential}
below, and hence they, rather than the sub-$\sigma$-algebras, occur 
in the title of the present paper.
 
Completeness as a tool for statistical theory was introduced systematically
by Leh\-mann and Scheff\'e~(1947, 1950, 1955, 1956), after special cases
had been considered before by 
Wald~(1942, 1944), Scheff\'e~(1943), and Halmos~(1946). Its classical use in 
estimation or testing theories is well-known, see Lehmann and Casella~(1998), 
Lehmann and Romano~(2005), and Pfanzagl~(1994) for textbook treatments,
and Matt\-ner and Mattner~(2013, Lemma 4.2) for a simple recent example in 
a rather applied setting. Again recently,
completeness has also been used in the econometric literature for studying
identifiability problems in instrumental regression models, see for example
D'Haultfoeuille~(2011).

However, for any given model and sub-$\sigma$-algebra, completeness can be 
difficult to verify even if strongly suspected. Hence sufficient criteria like 
the following main result of this paper can be useful.

\begin{Thm}   \label{Thm:main}
Let $\cP\subseteq\Prob(\cX,\cA)$ be a model and let $I$ be a set.
For each $i\in I$, let $\cC_i$ be a sub-$\sigma$-algebra of $\cA$ and
$(\cP_{i,\eta}:\eta \in \Eta_i)$ an exhaustion of $\,\cP$ with $\cC_i$ complete
sufficient for each $\cP_{i,\eta}$. Then $\bigvee_{i\in I} \cC_i$ 
is complete for $\cP$. 
\end{Thm}

Here, of course, $\bigvee_{i\in I} \cC_i$ denotes the supremum of 
$\{\cC_i:i\in I\}$ in the set of all sub-$\sigma$-algebras of $\cA$
partially ordered by inclusion.

The hypothesis~\eqref{Eq:Hypothesis_of_completeness} in the definition
of completeness entails that $h$ belongs to 
\[
 \cL^1(\cP) &:=& \bigcap_{P\in\cP}\cL^1(P),
\]
the set of all functions integrable with respect to every $P\in\cP$.
There are obvious analogues of Theorem~\ref{Thm:main}
and the other results in this paper 
involving $p$-completeness with some $p\in\mathopen]1,\infty\mathclose[$, 
where the implication 
\eqref{Eq:Hypothesis_of_completeness} $\Rightarrow$ 
\eqref{Eq:Conclusion_of_completeness}
is only required for $h\in \cL^p(\cP):=\bigcap_{P\in\cP}\cL^p(P)$,
or bounded completeness, which 
are not spelled out here except for one remark after Theorem~\ref{Thm:CKS}. 

We present two proofs for Theorem~\ref{Thm:main} in Sections~\ref{Sec:2} 
and~\ref{Sec:3} below. Examples and counterexamples are collected in 
Sections~\ref{Sec:Examples} and~\ref{Sec:Counterexamples}.
Let us proceed here by  stating  explicitly  the most 
transparent nontrivial special case of 
Theorem~\ref{Thm:main}, where $I=\{1,2\}$ and $\cP$ is 
parametrized by a cartesian product of two sets:

\begin{Cor}              \label{Cor:I={1,2}}
Let $\cP=\{P_\theta:\theta\in\Theta\}\subseteq\Prob(\cX,\cA)$ 
be a model with $\Theta=\Theta_1\times\Theta_2$ and  let
$\cC_1,\cC_2$ be sub-$\sigma$-algebras 
with these properties:
\begin{Aequiv}
\item For each $\theta_2\in \Theta_2$,  \label{Aequiv:Cor_theta2_fixed}
$\cC_1$ is complete sufficient for  
$\{P_{\theta_1,\theta_2}:\theta_1\in\Theta_1\}$.
\item For each $\theta_1\in \Theta_1$, \label{Aequiv:Cor_theta1_fixed}
$\cC_2$ is complete sufficient for 
$\{P_{\theta_1,\theta_2}:\theta_2\in\Theta_2\}$.
\end{Aequiv}
Then $\cC_1\vee\cC_2$ is complete for~$\cP$.  
\end{Cor}

One might rephrase for example assumption~\ref{Aequiv:Cor_theta2_fixed}
above as ``$\cC_1$ is partially complete sufficient for $\theta_1$'', 
hence the title of the present paper, and instead of  ``partially'', some 
would prefer ``profile''. 

Theorem~\ref{Thm:main} contains the classical 
Example~\ref{Example:Landers_Rogge} of 
Landers and Rogge~(1976), except for the latter's rather trivial 
``only if'' claim, which can not be added to Theorem~\ref{Thm:main}
by Counterexample~\ref{CounterExp:No_converse}. 
Similarly, the ``if'' claim of 
the special case of Example~\ref{Example:Landers_Rogge} 
where $I=\{1,2\}$ is contained in Corollary~\ref{Cor:I={1,2}},
in a result of R\"uschendorf~(1987, Lemma~1) concerning products involving 
Markov kernels, and in the following Theorem~\ref{Thm:CKS}. 
We recall that a statistical model is called \defn{homogeneous},
if its members are mutually absolutely continuous.

\begin{Thm}[essentially Cramer, Kamps, Schenk, 2002] \label{Thm:CKS} 
Let $\Theta_1,\Theta_2$ be sets and let 
$\cQ:= \{Q^{}_{\theta_1} : \theta_1 \in \Theta_1\}\subseteq\Prob(\cX_1,\cA_1)$,
$\cR := \{R_{\theta_1,\theta_2} : \theta_1\in\Theta_1,\theta_2\in\Theta_2\}
\subseteq\Prob(\cX_2,\cA_2)$, $\cP:=\{Q^{}_{\theta_1}\otimes 
R^{}_{\theta_1,\theta_2}:\theta_1\in\Theta_1,\theta_2\in\Theta_2\}$ 
be models  with these properties:
\begin{Aequiv}
 \item $\cQ$ is complete.  \label{Aequiv:CKS_Q_complete}
 \item For each $\theta_1\in\Theta_1$, \label{Aequiv:CKS_R_part_complete}
    $\{R_{\theta_1,\theta_2} :\theta_2\in\Theta_2\}$ is complete.
 \item For each $\theta_2\in\Theta_2$,   \label{Aequiv:CKS_homogeneous}
  $\{R_{\theta_1,\theta_2} :\theta_1\in\Theta_1\}$ is homogeneous.
 \item                                    \label{Aequiv:CKS_integrable}
  $\cL^{1}(\cP) =   \cL^1(\{Q_{\theta_1}\otimes R^{}_{\theta_1',\theta_2} :
 \theta_1,\theta_1'\in\Theta_1, \theta_2\in\Theta_2 \})$. 
\end{Aequiv}
Then $\cP$ is complete.
\end{Thm}

Theorem~\ref{Thm:CKS} is proved in Subexample~\ref{Subexample:CKS},
just after explaining how it implies Example~\ref{Example:Landers_Rogge} 
with $I=\{1,2\}$.

Except for a slightly different notation, Theorem~\ref{Thm:CKS}
in its present formulation differs from a claim of 
Cramer et al.~(2002, Theorem~2 and Remark~2) exactly by the addition of the 
integrability assumption~\ref{Aequiv:CKS_integrable}. The relevance of this 
assumption in the two proofs of Theorem~\ref{Thm:CKS} known to us
is explained by giving the new one of them in Subexample~\ref{Subexample:CKS} and 
commenting on the other one afterwards. Whether Theorem~\ref{Thm:CKS} would 
remain true if~\ref{Aequiv:CKS_integrable} were omitted seems to be unknown. 
Counterexample~\ref{CounterEx:Hom_CKS} shows that the homogeneity 
assumption~\ref{Aequiv:CKS_homogeneous} can not be omitted.
A version of Theorem~\ref{Thm:CKS} proposed by San Martin and Mouchart~(2007, 
Theorem~2.1) remains wrong even if assumption~\ref{Aequiv:CKS_integrable}
is added, see Counterexample~\ref{Couterex:San-M_Mouchart}. 

The problem with assumption~\ref{Aequiv:CKS_integrable} in 
Theorem~\ref{Thm:CKS} is a good illustration of the fact that bounded
completeness, where the implication  
\eqref{Eq:Hypothesis_of_completeness} $\Rightarrow$ 
\eqref{Eq:Conclusion_of_completeness} is required only for bounded
$\cC$-measurable functions $h$, is often much simpler to treat:
If we replace in Theorem~\ref{Thm:CKS} every ``complete'' by 
``boundedly complete'', then assumption~\ref{Aequiv:CKS_integrable}
can be omitted without substitute, as becomes clear by considering 
either proof of Theorem~\ref{Thm:CKS}.

To ease now the comparison of  Corollary~\ref{Cor:I={1,2}} 
with Theorem~\ref{Thm:CKS}, let us rewrite the latter in the style of the 
former, while reformulating part of the hypothesis using Basu 
theorems due to Basu~(1955) and Kagan~(1966).

\begin{Thm}[a rewrite of Theorem~\ref{Thm:CKS}]  \label{Thm:CKS_rewrite}
Let $\cP=\{P_\theta:\theta\in\Theta \}\subseteq\Prob(\cX,\cA)$ be a model
with $\Theta=\Theta_1\times\Theta_2$, and let  
$\cC_1,\cC_2$ be sub-$\sigma$-algebras with these properties:
\begin{Aequiv}
 \item For each $\theta_2\in\Theta_2$,  $\cC_1$ is complete for 
  $\{P_{\theta_1,\theta_2}:\theta_1\in\Theta_1\}$.
 \item                             \label{Aequiv:CKS_rewrite_theta1_fixed}
  For each $\theta_1\in\Theta_1$, $\cC_1$ is ancillary and $\cC_2$ is
  complete sufficient for   $\{P_{\theta_1,\theta_2}:\theta_2\in\Theta_2\}$.
 \item For each $\theta_2\in\Theta_2$, \label{Aequiv:CKS_rewrite_homogeneous} 
 $\{P_{\theta_1,\theta_2}|_{\cC_2}:\theta_1\in\Theta_1\}$ is homogeneous.
 \item             \label{Aequiv:CKS_rewrite_integrable}
  $\cL^1(\left\{P_\theta|_{\cC_1}\otimes P_\theta|_{\cC_2}:\theta\in\Theta\right\})
   = \cL^1(\left\{P_\theta|_{\cC_1}\otimes P_{\theta'}|_{\cC_2}:
    \theta,\theta'\in\Theta\right\})$. 
\end{Aequiv}
Then $\cC_1\vee\cC_2$ is complete for~$\cP$.  
\end{Thm}

Thus Theorem~\ref{Thm:CKS_rewrite} has, in comparison  
to Corollary~\ref{Cor:I={1,2}}, the  advantage of 
no sufficiency condition on $\cC_1$, 
but the disadvantage of the ancillarity condition on $\cC_1$
in assumption~\ref{Aequiv:CKS_rewrite_theta1_fixed}
and the additional assumptions~
\ref{Aequiv:CKS_rewrite_homogeneous},\ref{Aequiv:CKS_rewrite_integrable}.
Of course, as for Theorem~\ref{Thm:CKS}, it appears unknown whether 
assumption~\ref{Aequiv:CKS_rewrite_integrable} may be omitted in 
Theorem~\ref{Thm:CKS_rewrite}. 

While it  seems to us that Corollary~\ref{Cor:I={1,2}} is more 
frequently applicable than  Theorem~\ref{Thm:CKS} in either formulation,
one might try to look for  a natural common generalization.
Counterexample~\ref{CounterExp:Suff_needed}
shows that it is not possible to just omit the sufficiency assumption 
concerning $\cC_1$ in Corollary~\ref{Cor:I={1,2}}, even if the 
conditions~\ref{Aequiv:CKS_rewrite_homogeneous} 
and~\ref{Aequiv:CKS_rewrite_integrable} of Theorem~\ref{Thm:CKS_rewrite}
were added. In other words: Theorem~\ref{Thm:CKS_rewrite} would become 
false if the ancillarity condition in its 
assumption~\ref{Aequiv:CKS_rewrite_theta1_fixed} were omitted.

As remarked in the previous paragraph,
even in the special situation of Corollary~\ref{Cor:I={1,2}}, sufficiency may 
not be omitted in the hypothesis. Without any additional 
assumption, it may neither 
be added in the conclusion by Counterexample~\ref{CounterEx:C1_vee_C2_insuff}.
This suggests that it  should be impossible to state Theorem~\ref{Thm:main} 
just for the case of $I=\{1,2\}$ and refer to a simple  induction argument for 
the case of a general finite $I$. It further shows that the homogeneity 
assumption in Kagan's~(2006, Theorem~2.1) sharpening of the factorization
theorem can not be omitted. 
Assuming then homogeneity and a certain connectedness 
property of our exhaustions, we get the following
result, of which, to our surprise, we could not even find its part (a) in 
the literature.

\begin{Thm}                            \label{Thm:Main_hom_con} 
Let $\cP\subseteq\Prob(\cX,\cA)$ be an homogeneous model
and let $I$ be a set. For each $i\in I$, let $\cC_i$ be a 
sub-$\sigma$-algebra of $\cA$ and $(\cP_{i,\eta}:\eta\in\Eta_i)$ an exhaustion
of $\cP$. Assume that the following property holds:

If $P',P''\in\cP$, then there exist 
$n\in\N$ and $P_1,\ldots,P_n\in\cP$
with $P_1=P'$, $P_n=P''$, and such that for each $k\in\{1,\ldots,n-1\}$
there exist $i\in I$ and  $\eta\in\Eta_i$ with $P_k,P_{k+1}\in\cP_{i,\eta}$.
 
\smallskip\noindent{\rm\textbf{(a)}}
For each $i$, let $\cC_i$ be sufficient for each $\cP_{i,\eta}$.
Then $\bigvee_{i\in I}\cC_i$ is sufficient for $\cP$.

\smallskip\noindent{\rm\textbf{(b)}}
For each $i$, let $\cC_i$ be minimal sufficient for each $\cP_{i,\eta}$.
Then $\bigvee_{i\in I}\cC_i$ is minimal sufficient for $\cP$.

\smallskip\noindent{\rm\textbf{(c)}}
For each $i$, let $\cC_i$ be complete sufficient for each $\cP_{i,\eta}$.
Then $\bigvee_{i\in I}\cC_i$ is complete sufficient for $\cP$.
\end{Thm}

Note that the connectedness assumption concerning the exhaustions
in  Theorem~\ref{Thm:Main_hom_con} in particular holds if $I$ is finite, 
$\cP=\{ P_\theta:\theta\in\bigtimes_{i\in I}\Theta_i\}$,
$\Eta_i=\bigtimes_{j\in I\setminus\{i\}}\Theta_j$,
and each $\cP_{i,\eta}$ is obtained by fixing all but the $i$th of 
the coordinates of $\theta$ to coincide with those of $\eta$,
as in Corollary~\ref{Cor:I={1,2}} where $I=\{1,2\}$.
In the latter case, by an obvious modification of the proof of 
Theorem~\ref{Thm:Main_hom_con} given below, the hypothesis 
in~\ref{Thm:Main_hom_con}(a) can in fact be weakened to assuming
sufficiency of $\cC_1$ for $\{P_{\theta_1,\theta_2}: \theta_1\in\Theta_1\}$
for each $\theta_2\in\Theta_2$ as before, but  sufficiency of $\cC_2$ for  
$\{P_{\theta_1,\theta_2}: \theta_2\in\Theta_2\}$ for just one $\theta_1\in\Theta_1$.

Let us finally mention that the paper of Oosterhoff and Schriever~(1987)
contains many interesting examples loosely related to the topic of the  present paper.
\section{Proofs: Theorem~\ref{Thm:main} from scratch,
equivalence of Theorems~\ref{Thm:CKS} and~\ref{Thm:CKS_rewrite},
Theorem~\ref{Thm:Main_hom_con}}                                 \label{Sec:2}
Here we present our first proof for Theorem~\ref{Thm:main}, which generalizes
the original proof 
of Example~\ref{Eample:Landers_Rogge} without being any longer.

\begin{proof}[First proof of Theorem~\ref{Thm:main}] Let 
$h:\cX\rightarrow\R$ be 
measurable with respect to $\cC:=\bigvee_{i\in I} \cC_i$ 
and satisfy~\eqref{Eq:Hypothesis_of_completeness}.

Let $i\in I$. For $\eta\in\Eta_i$, the sufficiency of $\cC_i$ for 
$\cP_{i,\eta}$ yields a $g_\eta \in \bigcap_{P\in\cP_{i,\eta}}P(h \,\pmb|\, \cC_i)$, 
for  which~\eqref{Eq:Hypothesis_of_completeness} yields 
\[
  P g_\eta &=& 0 \quad\text{ for } P \in \cP_{i,\eta}
\]
and hence, using the completeness of $\cC_i$ for $\cP_{i,\eta}$, 
\[
  g_\eta &=& 0 \quad \cP_{i,\eta}\text{-a.s.},
\]
and thus $ P\1_C h=  P \1_C g_\eta = 0$ for $C\in\cC_i$ and $P \in \cP_{i,\eta}$. 
Hence, since $(\cP_{i,\eta}:\eta \in \Eta_i)$
exhausts $\cP$, for every $C\in\cC_i$ the 
assumption~\eqref{Eq:Hypothesis_of_completeness} also 
holds with $h$ replaced by $\1_Ch$.

Inductively repeating the above argument for different $i$ yields
\[                          \label{Eq:Hypothesis_of_completeness_with_C}
\qquad   P \1_E  h &=& 0 \quad\text{ for } 
   E\in \left\{\bigcap_{i\in I_0}C_i : I_0\subseteq I\text{ finite, }
  C_i\in\cC_i \right\}=:\cE
  \text{ and } P\in\cP
\]
and hence, as  $\cE$ is a $\cap$-stable generator of $\cC$
with $\cX\in\cE$,  $h=0$ $P$-a.s.~for $P\in\cP$.
\end{proof}

\begin{proof}[Proof that Theorems~\ref{Thm:CKS} and~\ref{Thm:CKS_rewrite}
are equivalent]
There is clearly no loss of generality in assuming 
$\cC_1\vee \cC_2=\cA$ in Theorem~\ref{Thm:CKS_rewrite}, which we will do
in this proof. Using  Basu~(1982, Theorems 1 and 3), with the first cited 
theorem due to Basu~(1955, Theorem 2) and the second due to Kagan~(1966)
and also proved by Barra~(1971, see Theorem 3 on pp.~26--27 of the 1981 English
edition),  we note that 
condition~\ref{Thm:CKS_rewrite}\ref{Aequiv:CKS_rewrite_theta1_fixed}
is equivalent to 
\smallskip
\begin{Aequiv}
 \item[(ii$'$)] $\cC_1$, $\cC_2$ are $\cP$-independent and,  
  for each $\theta_1\in\Theta_1$, $\cC_1$ is ancillary and $\cC_2$ is
  complete for   $\{P_{\theta_1,\theta_2}:\theta_2\in\Theta_2\}$.
\end{Aequiv}

\smallskip\noindent
Although not needed here, let us mention that, if we 
assume~\ref{Thm:CKS_rewrite}\ref{Aequiv:CKS_rewrite_homogeneous}, then
(ii$'$) is further equivalent to    
\begin{Aequiv}
\smallskip
\item[(ii$''$)] $\cC_1$, $\cC_2$ are $\cP$-independent and,  
  for each $\theta_1\in\Theta_1$, $\cC_2$ is
  complete sufficient for   $\{P_{\theta_1,\theta_2}:\theta_2\in\Theta_2\}$.
\end{Aequiv}

\smallskip\noindent
To check this, one can apply  Basu~(1982, Theorem 2), 
say in the version of Koehn and Thomas~(1975, Corollary),
noting that there the non-splitting assumption may equivalently be 
imposed on the model restricted to the sufficient $\sigma$-algebra,
as in Basu's~(1958) original version. 

Theorem~\ref{Thm:CKS} follows from Theorem~\ref{Thm:CKS_rewrite}
with (ii$'$) in place of~\ref{Aequiv:CKS_rewrite_theta1_fixed}
by letting $\cC_i$ denote the $\sigma$-algebra generated
by the $i$th coordinate projection in $(\cX_1\times\cX_2,\cA_1\otimes\cA_2)$.
 
Conversely Theorem~\ref{Thm:CKS_rewrite} with (ii$'$) in place 
of~\ref{Aequiv:CKS_rewrite_theta1_fixed} follows from Theorem~\ref{Thm:CKS} 
with $Q^{}_{\theta_1}:=P_{\theta_1,\theta_2}|_{\cC_1}$ and 
$R_{\theta_1,\theta_2} := P_{\theta_1,\theta_2}|_{\cC_2}$ for $\theta\in\Theta$, 
and by observing that  every $\cC_1\vee\cC_2$-measurable
$h:\cX\rightarrow\R$ is of the form 
$h(x)=g(f_1(x),f_2(x))$ for $x\in\cX$ with a 
$\cC_1\otimes\cC_2$-measurable function $g:\cX\times\cX\rightarrow \R$
and $f_i$ denoting the identity from $(\cX,\cA)$ to $(\cX,\cC_i)$.
\end{proof}
\begin{proof}[Proof of Theorem~\ref{Thm:Main_hom_con}] 
By homogeneity, there is a $\sigma$-finite measure $\mu$ on $(\cX,\cA)$ such 
that each $P\in\cP$ has some $\mathopen]0,\infty\mathclose[$-valued 
$\mu$-density $f_P$. Then, by Bahadur's~(1954, Section 6) version of a result 
of Lehmann and Scheff\'e~(1950, Section~6) as presented in 
Torgersen~(1991, p.~69, Theorem~1.5.9), the $\sigma$-algebra 
\[
  \cC &:=&\sigma(\cF)\quad \text{ with }
  \cF\,\,\,:=\,\,\,\left\{\frac{f^{}_{P'}}{f^{}_{P''}} \,:\,P',P''\in\cP\right\}
\]
is minimal sufficient for $\cP$, and, for each $i\in I$ and $\eta\in\Eta_i$,
 \[
  \cC_{i,\eta} &:=&\sigma(\cF_{i,\eta})\quad \text{ with }
  \cF_{i,\eta}\,\,\,:=\,\,\,\left\{\frac{f^{}_{P'}}{f^{}_{P''}} \,:\,P',P''\in\cP_{i,\eta}\right\}
\]
is minimal sufficient for $\cP_{i,\eta}$. 

(a) For each $i\in I$, the sufficiency assumption on $\cC_i$ yields for each $\eta\in \Eta_i$
first $\cC_{i,\eta} \subseteq \cC_i$ $[\cP_{i,\eta}]$, and then by homogeneity 
even  $\cC_{i,\eta} \subseteq \cC_i$ $[\cP]$, and hence we get 
\[
 \cC'&:=&  \bigvee_{i\in I}\bigvee_{\eta\in\Eta_i}\cC_{i,\eta}
 \,\,\,\subseteq\,\,\,\bigvee_{i\in I}\cC_i \quad [\cP].
\]

Let now $g=f^{}_{P'}/f^{}_{P''}\in \cF$. Choose $n\in\N$ and  $P_1=P',\ldots,P_n=P''$
as assumed to exist. Then $ g = \prod_{k=1}^{n-1} f^{}_{P_k}/f^{}_{P_{k+1}}$
is a product of functions each belonging to some $\cF_{i,\eta}$, and hence
$g$ is $\cC'$-measurable. Thus $\cC\subseteq \cC'$. Hence $\bigvee_{i\in I}\cC_i$
inherits sufficiency for $\cP$ from  its almost sure sub-$\sigma$-algebra $\cC$. 
 
(b) Keeping the notation of part (a), we also have $\cC'\subseteq \cC$ trivially
and hence $\cC=\cC'$.  
The stronger minimal sufficiency assumption on $\cC_i$ now even yields
$\cC_{i,\eta} = \cC_i$ $[\cP]$ for each $i$ and $\eta$, and hence
$\cC'= \bigvee_{i\in I}\cC_i$ $[\cP]$. Hence $ \bigvee_{i\in I}\cC_i$ inherits
minimal sufficiency for $\cP$ from $\cC$.
 
(c) Clear by combining part (a) with Theorem~\ref{Thm:main}.
\end{proof}

\section{A shorter proof of Theorem~\ref{Thm:main} via
optimal unbiased estimation}                                \label{Sec:3}
In this section, we give our second and very short proof of 
Theorem~\ref{Thm:main}
by using what we 
regard as the main version of the
basic theory of optimal  mean unbiased estimation,
for univariate estimands, as completed by Schmetterer and Strasser~(1974).
In spite of its conciseness and elegance,  
this theory in its entirety appears to be not widely known, 
and it is indeed not presented even in the union of 
the books on mathematical statistics we are aware of and which, like  
Schmetterer~(1974), Strasser~(1985), Witting~(1985), 
Pfanzagl~(1994), Witting and M\"uller-Funk~(1995),
Lehmann and Casella~(1998), and Bahadur~(2002),   
treat unbiased estimation  more thoroughly than others. 
Hence we proceed to give a brief summary in Theorem~\ref{Thm_RPLSSS} below.

\begin{Def}  \label{Def:Opt_unb}
Let $\cP\subseteq\Prob(\cX,\cA)$ be a model and
\[
  \cE &:=& \cL^1(\cP) \,\,\,=\,\,\,\bigcap_{P\in\cP}\cL^1(P)
\]
be the vector space of all measurable functions $g:\cX\rightarrow \R$ 
being integrable with respect to every $P\in\cP$. Then, for any function
$\kappa:\cP\rightarrow\R$, the elements of 
$\cE_{\kappa}:=\{ g\in\cE : Pg = \kappa(P)\text{ for }P\in\cP\}$ 
are called \defn{unbiased estimators} of the \defn{estimand} $\kappa$, and a 
$\hat{\kappa}\in\cE_\kappa$ is called \defn{optimal unbiased for $\kappa$}, if  
\[
 P\,\phi\circ(\hat{\kappa}-\kappa(P)) &\le& P\,\phi\circ(g-\kappa(P))
 \quad\text{ for } g\in\cE_\kappa \text{ and } P\in\cP 
\]
holds 
for every convex function $\phi:\R\rightarrow \R$. Finally, 
a $g\in\cE$ is called \defn{optimal unbiased}, without reference to 
 any estimand, if $g$ is optimal unbiased for its own 
expectation $P\mapsto Pg$.
\end{Def}

\begin{Thm}[Rao, Blackwell, Lehmann, Scheff\'e, Bahadur, Schmetterer, Strasser]
 \label{Thm_RPLSSS}
Let $\cP\subseteq\Prob(\cX,\cA)$ be a model and let
\[
  \cO &:=& \{A\in\cA : P\1_Ah=0 \text{ for }h\in\cE_0\text{ and }P\in\cP\},
\]
where $\cE_0$ is as in Definition~\ref{Def:Opt_unb} with $\kappa=0$.

\smallskip\noindent{\rm\textbf{(a)}}
$\cO$ is a sub-$\sigma$-algebra of $\cA$ and contains all $\cP$-null sets.

\smallskip\noindent{\rm\textbf{(b)}}
An estimator $\hat{\kappa}\in\cE$ is optimal unbiased iff it is 
$\cO$-measurable, and this is the case  iff 
$\hat{\kappa}\in \bigcap_{P\in\cP}P( \tilde{\kappa}  \,\pmb|\,\cO)$
holds for every $\tilde{\kappa}\in\cE$ with $P\hat{\kappa}=P\tilde{\kappa}$
for $P\in\cP$. 

\smallskip\noindent{\rm\textbf{(c)}}
$\cO$ is  (Lehmann-Scheff\'e-)complete. 
If $\cC\subseteq \cA$ is sufficient, then  $\cO\subseteq\cC$ $[\cP]$.

\smallskip\noindent{\rm\textbf{(d)}}
The following statements are equivalent:
\begin{Aequiv}
\item Every           \label{Equiv:estimable_optimal}
unbiasedly estimable parameter has an optimal unbiased estimator.
\item There exists a complete sufficient sub-$\sigma$-algebra. 
                                    \label{Equiv:Exists_compl_suff}
\item $\cO$ is sufficient.  \label{Equiv:cO_sufficient}
\end{Aequiv}
If these statements are true, then every complete sufficient 
sub-$\sigma$-algebra $\cC$ satisfies $\cC=\cO$ $[\cP]$ and $\cC\subseteq\cO$.
\end{Thm}
\begin{proof}
(a) $\cO$ is a Dynkin system and, since $A\in\cO$ implies $\1_Ah\in\cE_0$
for $h\in\cE_0$, also $\cap$-stable. The null set claim is trivial. 

(b) The first ``only if'' follows from Schmetterer and Strasser~(1974, Satz~2,
the special case of $p=1$) applied to, say, $W(t):=|t|-\log(1+|t|)$ for 
$t\in\R$. The second ``only if'' is clear since the definition of $\cO$
applied to $h:=\tilde{\kappa}-\hat{\kappa}$ yields 
$0\in\bigcap_{P\in\cP}P( \tilde{\kappa}-\hat{\kappa}  \,\pmb|\,\cO)$.
Finally, the first property, namely optimality of $\hat{\kappa}$, 
follows from the last by the conditional Jensen inequality argument 
familiar from the proof of the Rao-Blackwell theorem.

(c) $\cO$ is complete by the uniqueness theorem for integrals.
If $\cC\subseteq\cA$ is sufficient  and $A\in\cO$,
then $\1_A$ is optimal and its Rao-Blackwellization with respect 
to $\cC$ is better, hence also optimal and hence equal to $\1_A$ 
almost surely, yielding $A\in\cC$ $[\cP]$.

(d) Schmetterer and Strasser~(1974, S\"atze 4 and 5, the special case of 
$p=1$).
\end{proof}

Thus $\cO$ above is the $\sigma$-algebra generated by all optimal unbiased  
estimators
in the model $\cP$, by  Theorem~\ref{Thm_RPLSSS}(b)  and by considering 
the estimators $\1_A$ with $A\in\cO$, so let us here briefly call $\cO$ the 
\defn{optimal} $\sigma$-algebra of $\cP$.

Key sources of Theorem~\ref{Thm_RPLSSS} include the ones leading to 
the Rao-Blackwell-Lehmann-Scheff\'e theorem in 1950, for which partial credit is 
also due to the noneponymous Halmos, Hodges, and Barankin, in view of  the 
references given by Pfanzagl~(1994, pp.~105, 106, 107). Afterwards,
a fundamental idea of Rao~(1952, pp.~30--31), first made 
rigorous by  Bahadur~(1957) and later more generally by Torgersen~(1988) in the 
mathematically inconvenient and practically less important 
setting of the UMVU theory, finally led to the present result in the hands
of Schmetterer and Strasser~(1974), after earlier work of themselves
and of Padmanabhan, Linnik, and Rukhin cited by them. 
Further developments include Bahadur~(1976), Kozek~(1988),
Kagan and Konikov~(2006),  and Kagan and Malinovsky~(2013). 

Turning now to exhaustions of models, there is a trivial but useful remark of
Bondesson~(1983), stated  here in the version of Pfanzagl~(1994, p.~108, 
Remark~3.2.8).
 
\begin{Lem} 
\label{Lem:Bo_1983} Let  
$(\cP_{\eta}:\eta \in \Eta)$ be an exhaustion of the model 
$\cP\subseteq\Prob(\cX,\cA)$. If 
$\hat{\kappa}$ is an optimal unbiased estimator in each of the submodels 
$\cP_\eta$, then so it is in $\cP$.
\end{Lem}
\begin{proof} Absolutely trivial by Definition~\ref{Def:Opt_unb} and the 
definition of ``exhaustion''. 
\end{proof}

Combining Theorem~\ref{Thm_RPLSSS} with Lemma~\ref{Lem:Bo_1983} yields the 
following result.

\begin{Thm} 
\label{Thm:Partial_opt_implies_opt}
Let $(\cP_{\eta}:\eta \in \Eta)$ be an exhaustion of the model $\,\cP$,
and let $\cO$ and $\cO_\eta$ respectively denote the optimal $\sigma$-algebras
of $\,\cP$ and $\,\cP_\eta$ for $\eta\in\Eta$. Then 
$\,\bigcap_{\eta\in\Eta} \cO_{\eta} \subseteq\cO$. 
\end{Thm}
\begin{proof} If $A\in\bigcap_{\eta\in\Eta} \cO_{\eta}$, then  Theorem~\ref{Thm_RPLSSS}(b)
yields that $\1_A$ is an optimal unbiased estimator in each of the models $\cP_{\eta}$,  
and hence in $\cP$ by Lemma~\ref{Lem:Bo_1983}, and hence is $\cO$-measurable by 
Theorem~\ref{Thm_RPLSSS}(b) again. 
Thus $\bigcap_{\eta\in\Eta} \cO_{\eta} \subseteq\cO$.
\end{proof}

Finally,  Theorems~\ref{Thm_RPLSSS} and~\ref{Thm:Partial_opt_implies_opt}
yield:

\begin{proof}[Second proof of Theorem~\ref{Thm:main}]
Let $\cO$ and $\cO_{i,\eta}$ 
denote the optimal $\sigma$-algebras
of $\cP$ and $\cP_{i,\eta}$ for $i\in I$ and $\eta\in\Eta_i$. Then
$\cC_i \subseteq  \cO_{i,\eta}$ whenever $\eta\in\Eta_i$,
by  Theorem~\ref{Thm_RPLSSS}(d) applied to $\cP_{i,\eta}$ and its complete 
sufficient sub-$\sigma$-algebra $\cC_i$.
Hence $\cC_i\subseteq\bigcap_{\eta\in\Eta_i} \cO_{i,\eta}\subseteq\cO$ for every $i\in I$,
by  Theorem~\ref{Thm:Partial_opt_implies_opt}, and hence 
$\bigvee_{i\in I}\cC_i\subseteq\cO$. 
As $\cO$ is complete for $\cP$ by  Theorem~\ref{Thm_RPLSSS}(c), so is its 
sub-$\sigma$-algebra $\bigvee_{i\in I}\cC_i$.
\end{proof}

\section{Examples, including a proof of Theorem~\ref{Thm:CKS}}
                                                       \label{Sec:Examples}
\begin{Example}[Product models, Landers and Rogge, 1976] 
                   \label{Eample:Landers_Rogge} \label{Example:Landers_Rogge}
Let $I$ be a set. For each $i\in I$, let 
$\cP_i\subseteq \Prob(\cX_i,\cA_i)$ be a model.
Then $\cP := \{\bigotimes_{i\in I}P_i : P_i\in\cP_i\text{ for }i\in I\}$
is complete iff each $\cP_i$ is complete. 
\end{Example}
\begin{proof} Let $(\cX,\cA)$ be the product of the 
$(\cX_i,\cA_i)$. For $i\in I$, let $\cC_i$ be the sub-$\sigma$-algebra
of $\cA$ generated by the $i$th coordinate projection 
$\pi_i:\cX\rightarrow\cX_i$, $\Eta_i:=\bigtimes_{j\in I\setminus\{i\}}\cP_j$, and
\[
 \cP_{i,\eta} &:=&\left\{\bigotimes_{k\in I}P_k : 
  P_i\in\cP_i,\, P_j=Q_j \text{ for }j\in I\setminus\{i\} \right\}
 \quad\text{ for }\eta=(Q_j:j\in I\setminus\{i\})\in\Eta_i.
\]

Assume that each $\cP_i$ is complete. Then, for $i\in I$,  
$(\cP_{i,\eta}:\eta\in\Eta_i)$ exhausts $\cP$ and,
for $\eta \in\Eta_i$, $\cC_i$ is complete for $\cP_{i,\eta}$, 
since \eqref{Eq:Hypothesis_of_completeness}
for $h$ $\cC_i$-measurable, and thus  $h=g\circ\pi_i$ for some 
$\cA_i$-measurable $g$, here yields 
$0=(\bigotimes_{k\in I}P_k)h=P_ig$ for $P_i\in\cP_i$,
hence $g=0$ $\cP_i$-a.s., hence $h=0$ $\cP$-a.s.,
and $\cC_i$ is sufficient for $\cP_{i,\eta}$ by 
Basu~(1982, Theorem~3), 
since $\cD_i:=\bigvee_{j\in I\setminus \{i\}}\cC_j$ is ancillary (under $\cP_{i,\eta}$)
and $\cC_i,\cD_i$ are independent with $\cC_i\vee \cD_i$ sufficient.
Hence $\bigvee_{i\in I}\cC_i=\cA$ is complete for $\cP$ by 
Theorem~\ref{Thm:main}. 

Assume that $\cP$ is complete and $i\in I$. In the uninteresting 
case where $\cP=\emptyset$, we then have $\cA=\{\emptyset,\cX\}$,
and hence $\cA_i= \{\emptyset,\cX_i\}$ and thus $\cP_i$ complete for each 
$i\in I$. If now $\cP\neq\emptyset$, $i\in I$, and
$P_ih=0$ for $P_i\in\cP_i$, then $P\,h\circ\pi_i=0$ for $P\in\cP$, 
hence $h\circ\pi_i=0$ $\cP$-a.s., and hence, using $\cP\neq\emptyset$,
$h=0$ $\cP_i$-a.s. 
\end{proof}

We recall that Example~\ref{Example:Landers_Rogge} is the basic tool for
proving complete sufficiency of ``the vector of order statistics''
in certain nonparametric models, see for example 
Mandelbaum and R\"uschendorf~(1987), Pfanzagl~(1994, p.~21), 
and Mattner~(1996, p.~1267), where also 
Od\'en and Wedel~(1975) 
should have been cited as explained in Mattner~(1999, p.~405). 

Let us also mention that it took some twenty years from the desire to have 
Example~\ref{Example:Landers_Rogge} at least for finite $I$, shining 
through analogous results involving a more restrictive  assumption of 
``strong completeness'' in  Lehmann and Scheff\'e~(1955, section 7) 
or Fraser~(1957, p.~26), to the the proof of Landers and Rogge~(1976),
and that even the analogue involving bounded completeness
was provided only about one year earlier
by Plachky~(1975) with a somewhat complicated proof.

The special case of $I=\{1,2\}$ of the ``if''-statement of   
Example~\ref{Example:Landers_Rogge} is contained in 
Theorem~\ref{Thm:CKS}, as essentially 
already remarked by Cramer et al.~(2002):
If, in the notation of Theorem~\ref{Thm:CKS}, the laws do not actually 
depend on $\theta_1$, then conditions~\ref{Aequiv:CKS_homogeneous}  
and~\ref{Aequiv:CKS_integrable} are trivially fulfilled, 
and assumption~\ref{Aequiv:CKS_R_part_complete} is just the completeness of 
$\cR$. Conversely, but less obviously, one can  go the other way round:

\begin{Subexample}                \label{Subexample:CKS} 
Theorem~\ref{Thm:CKS} can be deduced from
Example~\ref{Example:Landers_Rogge}. 
\end{Subexample}

\begin{proof}
Under the assumptions of 
Theorem~\ref{Thm:CKS}, let 
$h:\cX_1\times\cX_2\rightarrow \R$ 
satisfy~\eqref{Eq:Hypothesis_of_completeness}. Then 
\[
   \int\left(\int h(x_1,x_2)\dd Q^{}_{\theta_1}(x_1)\right)
   \dd R_{\theta_1,\theta_2}(x_2)  &=&0 \quad\text{ for } \theta_1\in\Theta_1
  \text{ and }
\theta_2\in\Theta_2
\]
by Fubini. 
For each $\theta_1\in\Theta_1$, the completeness 
assumption~\ref{Aequiv:CKS_R_part_complete} yields 
\la     \label{Eq:Proof_CKS_1}
   \int h(x_1,x_2)\dd Q^{}_{\theta_1}(x_1)&=&0
\al
first for $\{R_{\theta_1,\theta_2} : \theta_2 \in\Theta_2\}$-a.e.~$x_2$,
and then, using the homogeneity assumption~\ref{Aequiv:CKS_homogeneous}, 
even for  
$\{R_{\theta_1',\theta_2} : \theta_1'\in\Theta_1, \theta_2 \in\Theta_2\}$-a.e.~$x_2$.
Hence   
\[        
 \int \left(\int h(x_1,x_2)\dd Q^{}_{\theta_1}(x_1)\right)
   \dd R_{\theta_1',\theta_2}(x_2) 
  &=& 0 \quad\text{ for } \theta_1,\theta_1'\in\Theta_1, \theta_2\in\Theta_2.
\]
Now, thanks to assumption~\ref{Aequiv:CKS_integrable}, the 
$\cP$-integrable function $h$ is also integrable with respect to each 
$Q^{}_{\theta_1}\otimes R_{\theta_1',\theta_2}$ and hence Fubini yields
\la          \label{Eq:Proof_CKS_2}
 \int h \dd Q^{}_{\theta_1} \otimes R_{\theta_1',\theta_2} 
  &=& 0 \quad\text{ for } \theta_1,\theta_1'\in\Theta_1, \theta_2\in\Theta_2
\al
and hence, applying Example~\ref{Example:Landers_Rogge} for $I=\{1,2\}$
to the models $\cQ$ and $\cR$, which are complete 
by~\ref{Aequiv:CKS_Q_complete} and~\ref{Aequiv:CKS_R_part_complete},
we get $h=0$ $Q^{}_{\theta_1} \otimes R_{\theta_1',\theta_2}$-a.s.~for
$\theta_1,\theta_1'\in\Theta_1$ and  $\theta_2\in\Theta_2$, and hence 
in particular~\eqref{Eq:Conclusion_of_completeness}.
\end{proof}

The proof of Cramer et al.~(2002, pp.~273--274), valid under the present 
additional assumption~\ref{Aequiv:CKS_integrable},  uses the original proof of 
Example~\ref{Example:Landers_Rogge}, rather than the result, and is hence 
a bit longer. On the other hand, starting from \eqref{Eq:Proof_CKS_1},
instead of having to conclude~\eqref{Eq:Proof_CKS_2}, they would 
only need to justify the second equality in 
\[
 0&=& \int_B\left( \int h(x_1,x_2)\dd Q^{}_{\theta_1}(x_1) \right)  
\dd   R_{\theta_1',\theta_2}(x_2) \\
 &=& \int\left( \int_B h(x_1,x_2)\dd   R_{\theta_1',\theta_2}(x_2)\right)  
\dd Q^{}_{\theta_1}(x_1) \quad\text{ for }B\in\cA_2,
\] 
which, in view of  an example of Fichtenholz~(1924), 
might also  hold for some functions $h$ not being 
$Q^{}_{\theta_1}\otimes R_{\theta_1',\theta_2}$-integrable.

\begin{Example} Completeness for multiparameter exponential models,
with natural parameter spaces with nonempty interiors, follows
from the one-parameter case.
\end{Example}
\begin{proof} Let $\cP=\{P_\alpha:\alpha\in\Alpha\}$ be a $k$-parameter
exponential model with natural parameter space $\Alpha\subseteq\R^k$,
that is, each $P_\alpha$ has a $\mu$-density
$x\mapsto c(\alpha)h(x)\exp(\sum_{i=1}^k \alpha_iT_i(x))$, 
and let $\Alpha_0\subseteq\Alpha$ be nonempty and open.
Then Theorem~\ref{Thm:main}, applied to $\cP_0:=\{P_\alpha:\alpha\in\Alpha_0\}$,
$I:=\{1,\ldots,k\}$, $\cC_i:=\sigma(T_i)$, and
the $\cP_{i,\eta}$ being the one-parameter exponential models
obtained from $\cP_0$ by varying $\alpha_i$ while keeping  all other parameter 
coordinates fixed,
yields the completeness of $\bigvee_{i\in I}\cC_i$ for $\cP_0$ and hence, 
by homogeneity of $\cP$, also  for $\cP$.
\end{proof}

The above reduction from the multiparameter to the one-parameter case 
appears to be shorter than the one in Pfanzagl~(1994, pp.~26--27).

The next lemma and its corollary, provided here in preparation for 
Example~\ref{Example:Compl_suff_unkknow_trunc}
and Subexample~\ref{Subexample:Trunc_exp_fam}, contain in particular 
the determination of complete sufficient statistics in 
discrete as well as continuous and even mixed 
``taxi problem models'' with unknown 
lower and upper bounds, compare Feller~(1970, Example~(e) on p.~226
and the exercises~8,9 on pp.~237--238) for the discrete case, 
without having to calculate the joint laws of 
sample minima  and maxima as apparently intended  in 
Lehmann and Casella~(1998, Problem~6.30 on p.~72).

If $(\cX,\cA,\mu)$ is a measure space and $E\in\cA$ with $0<\mu(E)<\infty$,
then we consider the conditional  law 
$\mu(\,\cdot\, \pmb{|}E ):= \mu(\cdot\cap E)/\mu(E)$.
 
\begin{Lem}   \label{Lem:uniform_cap-stable_compl_suff}
Let $(\cX,\cA,\mu)$ be a measure space,   $\cE\subseteq\cA$, 
$n\in\N$, 
$\cP:=\{\mu(\,\cdot\, \pmb{|}E )^{\otimes n} : E \in\cE, 0<\mu(E)<\infty\}$,
and $\cC:=\sigma(E^n:E\in\cE)$.

\smallskip\noindent{\rm\textbf{(a)}} 
Let $\mu$ be $\sigma$-finite. Then $\cC$ is sufficient for $\cP$. 

\smallskip\noindent{\rm\textbf{(b)}} 
Let $\cE$ be  $\cap$-stable. Then $\cC$ is complete for $\cP$.
\end{Lem}
\begin{proof} 
Let $\cE_0 := \{E\in\cE :0<\mu(E)<\infty\}$ and
$P_E := \mu(\,\cdot\, \pmb{|}E )^{\otimes n} $ for $E\in\cE_0$.

(a) For $E\in\cE_0$, the function $\frac1{\mu(E)^n}\1_{E^n}$
is a  $\cC$-measurable density of $P_E$, with respect to the 
$\sigma$-finite measure $\mu^{\otimes n}$.
Hence $\cC$ is sufficient for $\cP$ by the factorization criterion.  

(b) Let $h:\cX\rightarrow \R$ be $\cC$-measurable 
with~\eqref{Eq:Hypothesis_of_completeness}. Let $E_0\in\cE_0$.
For $E\in\cE$, we then have
\[
 \int_{(E\cap E_0)^n}h\dd P^{}_{E_0} 
  &=& \frac1{\mu(E_0)^n}\int h 
 \dd \left(\1^{}_{E\cap E_0}\,\mu(\,\cdot\, \pmb{|}E_0 )\right)^{\otimes n}
 \,\,\,=\,\,\, 0,
\]
namely trivially if $\mu(E\cap E_0)=0$, and otherwise 
by~\eqref{Eq:Hypothesis_of_completeness}, since then
$\left(\1_{E\cap E_0}\mu(\,\cdot\, \pmb{|}E )\right)^{\otimes n}
= \mu(E\cap E_0)^nP_{E\cap E_0}$ and $P_{E\cap E_0} \in \cP$.
Now on $E_0^n$, 
$\{(E\cap E_0)^n : E\in \cE\}$ is a  $\cap$-stable generator 
of the trace of  $\cC$, and contains $E_0^n$, 
and hence we get $h=0$ $P_{E_0}$-a.s. 
\end{proof}

By  Counterexample~\ref{Counterex:sigma-finiteness_needed},
the above  $\sigma$-finiteness assumption can not be omitted.

In 
\ref{Cor:Taxi_examples}, \ref{Example:Intervals_vs_rays}  
and \ref{Subexample:Trunc_exp_fam} below,
$\cX$ is assumed to be a  subset of the extended real line $\overline{\R}$. 
We then call a set $J\subseteq \cX$ an \defn{upray in $\cX$}, if
$x\in J, y\in \cX, x\le y$ jointly imply $y\in J$. We analogously 
define \defn{downray in $\cX$}. And we call $I$ an \defn{interval in $\cX$},
if $x,z\in I, y\in \cX, x\le y\le z$ jointly imply $y\in I$.
Clearly, in $\cX$, every ray is an interval, and $I$ is an interval iff 
$I = J\cap K$ for an upray $J$ and a downray $K$.

If $(X_i:i\in I)$ is a family of functions $X_i$, all  with the same domain of 
definition $\Omega$, then we write $\lpp X_i : i\in I\rpp$ for the function
$\Omega \ni \omega \mapsto (X_i(\omega):i\in I)$. 
  
\begin{Cor}[Completeness of $\min$ and $\max$ in truncation models]
                                          \label{Cor:Taxi_examples} 
Let $\mu$ be a  measure on a measurable subspace $(\cX,\cA)$ of 
$\overline{\R}$, $n\in\N$, and  $\lpp X_1,\ldots,X_n \rpp := \id_{\cX^n}$.

\smallskip\noindent{\rm\textbf{(a)}} 
$\min_{i=1}^nX_i$ is complete sufficient for 
$\{\mu(\,\cdot\,\pmb{|}J)^{\otimes n}:J\text{ upray in }\cX, 0<\mu(J)<\infty\}$.

\smallskip\noindent{\rm\textbf{(b)}} 
$\max_{i=1}^nX_i$ is complete sufficient for 
$\{\mu(\,\cdot\,\pmb{|} K )^{\otimes n}:
K\text{ downray in }\cX, 0<\mu(K)<\infty\}$.

\smallskip\noindent{\rm\textbf{(c)}} $\lpp  \min_{i=1}^nX_i,\max_{i=1}^nX_i\rpp$ 
is complete for 
\[
 \{\mu(\,\cdot\,\pmb{|}I)^{\otimes n}:I\text{ interval in }\cX,0<\mu(I)<\infty\}
\]
and, if $\mu$ is assumed to be $\sigma$-finite, also sufficient.
\end{Cor}
\begin{proof} 
We apply Lemma~\ref{Lem:uniform_cap-stable_compl_suff} with, respectively, 

(a) $\cE$ $:=$ set of all uprays in $\cX$ and 
$\sigma(E^n:E\in\cE) =\sigma((\min_{i=1}^nX_i)^{-1}(E): E \in\cE) 
= \sigma(\min_{i=1}^nX_i )$, where the last identity holds since $\cE$ 
generates $\cA$,

(b) $\cE$ $:=$ set of all downrays in $\cX$ and 
$\sigma(E^n:E\in\cE) =\sigma(\max_{i=1}^nX_i )$,

(c) $\cE$ $:=$ set of all intervals in $\cX$ and 
\[ 
 \sigma(E^n:E\in\cE) &=& \sigma( (J\cap K)^n : J\text{ upray}, K \text{ downray})\\
 &=& \sigma (\lpp\min_{i=1}^nX_i ,\max_{i=1}^nX_i \rpp^{-1}(J\times K):
    J\text{ upray }, K \text{ downray})\\ 
 &=& \sigma( \min_{i=1}^nX_i ,\max_{i=1}^nX_i),
\]
and, in parts (a) and (b),  with $\mu$ replaced by 
the $\sigma$-finite measure $\mu(\cdot \cap \cX_0)$
with $\cX_0 := \bigcup \{E\in \cE : \mu(E)<\infty\}$.
\end{proof}

\begin{Example}  \label{Example:Intervals_vs_rays}    
The completeness assertion in Corollary~\ref{Cor:Taxi_examples}(c) 
also follows from~\ref{Cor:Taxi_examples}(a) and~\ref{Cor:Taxi_examples}(b)
via Theorem~\ref{Thm:main}, without using 
Lemma~\ref{Lem:uniform_cap-stable_compl_suff}(b).  
\end{Example}
\begin{proof} Let 
$ \cI :=\{ I\subseteq\cX : I\text{ interval in }\cX,\, 0<\mu(I)<\infty\}$ 
and
\[
 \Theta&:=&\{(J,K): J\text{ upray in }\cX,\,K\text{ downray in }\cX,\, 
 J\cap K\in\cI\}.
\]
For $(J,K)\in\Theta$, let 
$\pi_1(J,K):=J$ and $\pi_2(J,K):=K$. Let 
$\Eta_2,\Eta_1$, in this order, denote the images of the coordinate
projections $\pi_1,\pi_2$.
If $K\in \Eta_1$ is fixed, then 
$\min_{\nu=1}^n X_\nu$ is complete sufficient for 
$\cP_{1,K}:=\{\{\mu(\,\cdot\, \pmb{|}J\cap K )^{\otimes n}:  (J,K)\in\Theta\}$, 
by~\ref{Cor:Taxi_examples}(a)
with $K$ and $\mu(\cdot\cap K)$ in the roles of $\cX$ and $\mu$.
Analogously, if $J\in \Eta_2$ is fixed, then $\max_{\nu=1}^n X_\nu$ is 
complete sufficient for
$\cP_{2,J}:=\{\mu(\,\cdot\, \pmb{|}J\cap K )^{\otimes n} 
 : (J,K)\in\Theta\}$.
As each of $(\cP_{1,K}: K \in \Eta_1)$ and  $(\cP_{2,J}: J \in \Eta_2)$
is an exhaustion of the model $\cP$ of~\ref{Cor:Taxi_examples}(c), 
completeness of  $\lpp\min_{\nu=1}^n X_\nu,\max_{\nu=1}^n X_\nu\rpp $ follows from 
Theorem~\ref{Thm:main}.
\end{proof}

We next recall as Lemma~\ref{Lem:Compl_suff_weight} below part of a result 
of Smith~(1957) about weighted models, for which Patil~(2002) may serve 
as an introduction.  We provide a short proof for convenience, and also since 
our part (b), being slightly more general than the original, might appear to 
contradict the correct remark in Smith~(1957, p.~248, second line after 
Theorem). To this end,  we need the following perhaps not too 
well-known probabilistic property of conditional 
expectations implicitly proved by   Smith~(1957, p.~249).

\begin{Lem}[Strictness in  the isotonicity of conditional expectations, 
Smith 1957]                                               \label{Lem:Smith_C_E}
Let $X,Y$ be $\overline{\R}$-valued random variables
on $(\Omega,\cA,\P)$ with $X\le Y $ a.s.~and with existing expectations,
possibly infinite. Let $\cC$ be a sub-$\sigma$-algebra of $\cA$ 
and let $X_0\in\E( X{\pmb\,|\,}\cC)$ and   $Y_0\in\E( Y{\pmb\,|\,}\cC)$.
Then $X_0\le Y_0$ a.s. If in addition
neither $\E X=\E Y=\infty$ nor $\E X=\E Y=-\infty$, then  
$X_0 < Y_0$ a.s.~on $\{X<Y\}$.
\end{Lem}
\begin{proof} The first claim is of course standard, see 
e.g.~Hoffmann-J{\o}rgensen~(1994, p.~452). Under the additional 
assumption, $Y-X$ and $Y_0-X_0$ are defined a.s., and using
$\{X_0=Y_0\}\in\cC$ yields 
\[
 \E(Y-X)\1_{\{X_0=Y_0\}} &=&\E Y\1_{\{X_0=Y_0\}}-\E X\1_{\{X_0=Y_0\}} \\
 &=& \E Y_0\1_{\{X_0=Y_0\}}-\E X_0\1_{\{X_0=Y_0\}} \,\,\,=\,\,\, 0
\] 
and hence $X=Y$ a.s.~on $\{X_0=Y_0\}$, which, by contraposition
and since  $X\le Y$ and $X_0\le Y_0$ a.s., yields the second claim. 
\end{proof}
 
\begin{Lem}[Permanence of sufficiency and complete  
sufficiency under a fixed weighing, Smith 1957]   \label{Lem:Compl_suff_weight}
Let $\cP\subseteq\Prob(\cX,\cA)$ be a statistical model, 
$q:\cX\rightarrow [0,\infty[$ $\cP$-integrable,
and $\cP_q :=\{ P_q : P\in\cP, Pq >0\}$ with 
$P_q$ denoting the $q$-weighted version of $P$, 
given by    $P_q(A):=P\1_Aq/Pq$ for $A\in\cA$. 

\smallskip\noindent{\rm\textbf{(a)}} Let $\cC$ be sufficient for $\cP$. 
Then $\cC$ is sufficient for $\cP_q$.

\smallskip\noindent{\rm\textbf{(b)}} Let $\cC$ be complete sufficient for $\cP$.
Then $\cC$ is complete sufficient for $\cP_q$.
\end{Lem} 
\begin{proof} (a) There is an 
$f\in\bigcap_{P\in\cP}P(q\,\pmb{|}\,\cC)$. For $P_q\in\cP_q$, we then 
have $P_q(C)=P\1^{}_Cq /Pq =P\1^{}_Cf/Pq$ for $C\in\cC$, and so  
$f/Pq$ is a $P|_\cC$-density of $P_q|_\cC$. Now let $A\in\cA$ be given. 
With $g\in \bigcap_{P\in\cP}P(\1^{}_Aq\,\pmb{|}\,\cC)$,
we let $h:=g/f$ with $0/0:=0$, and, since $\1^{}_Aq \le q$,
we $\cP$-a.s.~have $g\le f$ and hence 
the implication
$f=0 \Rightarrow g=0$. Hence, for $P\in\cP_q$, 
we have 
$P_q\1^{}_C h = P\1^{}_C h f/Pq = P\1^{}_C g/Pq = P\1^{}_C \1^{}_Aq/Pq
= P_q\1^{}_C\1^{}_A$ for $C\in\cC$, 
and hence $h\in \bigcap_{P_q\in\cP_q}P_q(A\,\pmb{|}\,\cC)$.

(b) Let $h:\cX\rightarrow\R$ be $\cC$-measurable with $P_qh=0$ for
$P_q\in \cP_q$. Then, trivially for $P\in\cP$ with $Pq=0$
and hence for every $P\in\cP$, 
we have $0=P q h = Pgh$ with 
$g\in\bigcap_{P\in\cP}P(q\pmb{|}\cC)$ chosen by sufficiency. 
By completeness
of $\cC$ for $\cP$ and by $\cC$-measurability of $gh$, we have, for every 
$P\in\cP$, first $gh=0$ $P$-a.s.~and then, since $g>0$ $P$-a.s.~on $\{q>0\}$
by Lemma~\ref{Lem:Smith_C_E} applied to $X:=0$ and $Y:=q$, we have
$h=0$ $P$-a.s.~on $\{q>0\}$, and finally $h=0$ $\cP_q$-a.s. 
\end{proof}

We remark that Smith~(1957) considers also minimal sufficiency. 

\begin{Example}[Modification of complete sufficiency   
under an unknown truncation]         \label{Example:Compl_suff_unkknow_trunc}
Let $\cP_0\subseteq\Prob(\cX,\cA)$ be a model, 
$\cE\subseteq \cA$ $\cap$-stable,
and $P_E:= P(\,\cdot\, \pmb{|}E )$ for  $P\in\cP_0$ and  $E\in\cE$
with $P(E)>0$. Let 
$n\in\N$ and $\cC$ complete sufficient for $\{P^{\otimes n} : P\in \cP_0\}$.
Then $\cC\vee \sigma( E^n : E\in\cE)$ is complete sufficient for 
$ \{P_E^{\otimes n} : P\in\cP_0, E\in\cE, P(E)>0\}$.
\end{Example}
\begin{proof}
We apply Theorem~\ref{Thm:main} with $I:=\{1,2\}$:
Let $\Eta_1:=\cE$ and $\cP_{1,E}:=\{P_E^{\otimes n} : P\in\cP_0,P(E)>0\}$
for $E\in\Eta_1$. Then   $\cC_1:=\cC$ is complete sufficient for each 
$\cP_{1,E}$, by Lemma~\ref{Lem:Compl_suff_weight} applied to 
$\{P^{\otimes n} : P\in\cP_0\}$ 
and  $q:=\1_{E^n}$.  
Let  $\Eta_2:=\cP_0$ and 
$\cP_{2,P}:=\{ P_E^{\otimes n} : E\in\cE, P(E)>0\}$ for $P\in\cP_0$.
Then $\cC_2:= \sigma(E^n : E\in\cE)$ is complete sufficient for 
each $\cP_{2,P}$, by Lemma~\ref{Lem:uniform_cap-stable_compl_suff}
with $\mu=P$. Finally, each  $(\cP_{i,\eta}:\eta \in\Eta_i)$ exhausts $\cP$.
\end{proof}

\begin{Subexample}[Truncated exponential families on the line]
                                           \label{Subexample:Trunc_exp_fam}
Let $(\cX,\cA)$ be a measurable subspace of $\overline{\R}$ and let 
$\cP_0\subseteq\Prob(\cX,\cA)$ be a $k$-parameter exponential family 
in the statistics $T_1,\ldots,T_k:\overline{\R}\rightarrow \R$ and 
the natural parameters $a_1,\ldots,a_k:\cP_0\rightarrow \R$  
with  $\{ (a_1(P),\ldots,a_k(P)):P\in\cP_0\}$ having nonempty interior.
Let $n\in\N$, $\lpp X_1,\ldots,X_n \rpp := \id_{\cX^n}$, and 
$ S:=\lpp \sum_{i=1}^n T_1(X_i),\ldots,\sum_{i=1}^nT_k(X_i)\rpp$.

\smallskip\noindent{\rm\textbf{(a)}} 
$\lpp S, \min_{i=1}^nX_i\rpp$ is complete sufficient for 
\[ 
 \{P(\,\cdot\, \pmb{|} J )^{\otimes n} : P\in\cP_0, J\text{ upray in }\cX, P(J)>0\}.
\]
{\rm\textbf{(b}} 
$\lpp S, \max_{i=1}^nX_i\rpp$ is complete sufficient for 
\[
 \{P(\,\cdot\, \pmb{|} K )^{\otimes n} : P\in\cP_0, K\text{ downray in }\cX, P(K)>0\}.
\]
{\rm\textbf{(c)}} 
$\lpp S, \min_{i=1}^nX_i,  \max_{i=1}^nX_i\rpp$ is complete sufficient for 
\[
 \{P(\,\cdot\, \pmb{|} I )^{\otimes n} : P\in\cP_0, I\text{ interval in }\cX, P(I)>0\}.
\]
\end{Subexample}                            
\begin{proof} Example~\ref{Example:Compl_suff_unkknow_trunc} 
with $\cC:=\sigma(S)$ and Corollary~\ref{Cor:Taxi_examples}.
\end{proof}

\begin{Subsubexample}[Lehmann and Scheff\'e, 1955]   
Let  $X_1,\ldots,X_n$ be i.i.d.~according  \label{Subsubex:LS_exponential}
to the shifted exponential law, with $\pmb{\lambda}$-density given by 
$f_{a,b}(x)=\frac1b\exp(-\frac{x-a}b){\pmb1}_{]a,\infty[}(x)$ 
for $x\in\mathbb{R}$, where $a\in\R$ and $b\in\mathopen] 0,\infty\mathclose[$. 
Then $\min_{i=1}^nX_i$ and $\sum_{i=1}^nX_i$ are jointly complete 
sufficient.
\end{Subsubexample}
\begin{proof} Let us first replace ``$a\in\R$'' by ``$a>a_0$'' with 
$a_0\in\R$ fixed. Then the claim is a special case of 
\ref{Subexample:Trunc_exp_fam}(a), with $\cX=\R$, 
$\cP_0:=\{f_{a_0,b}\pmb{\lambda}: b\in\mathopen]0,\infty\mathclose[\}$,
$k=1$, $T_1(x)=x$. The claim with ``$a\in\R$'' then follows by writing 
the model in question as a union of the  
increasing sequence of the models with ``$a>a_0$'' with $a_0\in-\N$.
\end{proof}

The proof of Lehmann and Scheff\'e~(1955), given except for some 
measurability details also in Lehmann and Casella~\cite[p.~43]{LC_1998},
uses the independence of $X_{(1)}:=\min_{i=1}^nX_i$
and $\sum_{i=1}^n(X_i-X_{(1)})$.
Note that, nevertheless,  the Landers-Rogge 
theorem Example~\ref{Eample:Landers_Rogge} would not apply, 
since the law of $X_{(1)}$ depends on both parameters $a$ and $b$.
On the other hand, Theorem~\ref{Thm:CKS} does apply, 
as intended by Cramer et al.~(2002, p.~271), but to check 
condition~\ref{Thm:CKS}(iv), one apparently has to do some computation like 
the following: The densities of $\sum_{i=1}^n(X_i-X_{(1)})$ and $X_{(1)}$
are, up to constants depending only on the parameters $a$ and $b$, given
by $g_b(x):=x^{n-2}\exp(-\frac{x}b)\1_{]a,\infty[}(x)$ and 
$h_{a,b}(y):= \exp(-\frac{ny}b)\1_{]0,\infty[}(y)$,  and we  have
$ g_b(x)h_{a,b'}(y) \le g_{b\wedge b'}(x)h_{a,b\wedge b'}(y)$.

\section{Counterexamples}                    \label{Sec:Counterexamples}

We can not add an ``only-if''-statement in Theorem~\ref{Thm:main}, 
as present in its special case Example~\ref{Eample:Landers_Rogge}, not even 
in Corollary~\ref{Cor:I={1,2}} and with
``sufficiency'' strengthened to ``minimal sufficiency'' in the 
hypothesis:

\begin{Counterexample}  \label{CounterExp:No_converse}  
There exists a model $\cP=\{P_\theta:\theta\in\Theta_1\times \Theta_2\}$
with sub-$\sigma$-algebras $\cC_1$ and $\cC_2$ such that $\cC_1$ is minimal
sufficient but incomplete for each $\{P_{\theta_1,\theta_2}:\theta_1\in\Theta_1\}$
with $\theta_2\in\Theta_2$,   $\cC_2$ is minimal
sufficient but incomplete for each $\{P_{\theta_1,\theta_2}:\theta_2\in\Theta_2\}$
with $\theta_1\in\Theta_1$, and $\cC_1\vee\cC_2$ is complete sufficient for $\cP$.  
\end{Counterexample}
\begin{proof} We  may take $\Theta_1=\Theta_2=\mathopen]0,\infty\mathclose[$,
$n\in\N$, and $P_\theta$ the law of $n$ i.i.d.~normal random variables $X_i$, 
each with the density 
$\R\ni x\mapsto c(\theta)\exp(\theta^{}_1\theta^{}_2x-\theta_1^2\theta_2^3x^2)$,
and $\cC_1=\cC_2=\sigma(\sum_{i=1}^nX_i, \sum_{i=1}^n X_i^2)$.  
Here the first ``minimal sufficient but incomplete'' claim follows from 
the linear independence but algebraic dependence of 
$1,\theta^{}_1\theta^{}_2,-\theta_1^2\theta_2^3$ as functions of $\theta_1$,
with $\theta_2$ fixed, compare Pfanzagl~(1994, Theorem~1.6.9 and 
Wijsman's Theorem 1.6.23). The analogous second claim follows similarly.
Completeness of $\cC_1\vee\cC_2$ follows from the openness of
$\{(\theta^{}_1\theta^{}_2, -\theta_1^2\theta_2^3):\theta\in \Theta_1\times\Theta_2\} = \mathopen]0,\infty\mathclose[\times\mathopen]-\infty,0\mathclose[$. 
\end{proof}

Corollary~\ref{Cor:I={1,2}} would become false if any of the two sufficiency
assumptions were omitted, even if the assumptions 
(iii) and (iv) of Theorem~\ref{Thm:CKS_rewrite} were added:

\begin{Counterexample}             \label{CounterExp:Suff_needed}
Theorem~\ref{Thm:CKS_rewrite} would become false if the ancillarity condition
in its assumption~\ref{Aequiv:CKS_rewrite_theta1_fixed} were omitted. 
\end{Counterexample}
\begin{proof}
We may take $\cX:=\{0,1\}^2$, $\cA:=2^\cX$,
$\Theta_1:=\{0,1\}$, $\Theta_2:=\mathopen]0,\frac12\mathclose[$, and,
for $\theta\in\Theta$, 
\[
 P_\theta &:=& \left\{\begin{array}{ll}  \big((1-\theta_2)\delta_0
       +\theta_2\delta_1\big)^{\otimes 2}&\text{ if }\theta_1=0,\\
 \big(\theta_2\delta_0+(1-\theta_2)\delta_1\big)^{\otimes 2}&\text{ if }\theta_1=1,
 \end{array} \right.
\]
$\lpp X_1,X_2\rpp :=\id_\cX$, $\cC_1:=\sigma(X_1)$,
and $\cC_2:=\sigma(X_1+X_2)$.
The completeness of $\cC_1$ required by~\ref{Thm:CKS_rewrite}(i) 
is easily checked, the complete sufficiency of $\cC_2$  
in~\ref{Thm:CKS_rewrite}\ref{Aequiv:CKS_rewrite_theta1_fixed} is  
a standard result for 
Bernoulli chains, and the assumptions~\ref{Thm:CKS_rewrite}(iii),(iv) 
are obviously fulfilled. But $\cC_1\vee \cC_2 =\sigma(X_1,X_2)$ is 
incomplete, since we have $P_\theta\,(X_1-X_2) =0$ for each $\theta\in\Theta$. 
\end{proof}

\begin{Counterexample} \label{CounterEx:Hom_CKS} 
Theorem~\ref{Thm:CKS} would become false if 
its homogeneity assumption~\ref{Aequiv:CKS_homogeneous} were omitted.
\end{Counterexample}
\begin{proof}
Let $\cX_1:=\cX_2:=\Theta_1:=\Theta_2:=\{0,1\}$, $\cA_i$ the power set of 
$\cX_i$, $Q_0:=\frac13\delta_0+\frac23\delta_1$, 
$Q_1:=\frac23\delta_0+\frac13\delta_1$,
and $R_{\theta_1,\theta_2}:=\delta_{\theta_1}$ for $\theta\in\Theta_1\times\Theta_2$.
Then all assumptions of 
Theorem~\ref{Thm:CKS} but~\ref{Aequiv:CKS_homogeneous}
are fulfilled, and with $h(x_1,x_2):=|x_1-x_2|-\frac23$, 
we have~\eqref{Eq:Hypothesis_of_completeness} but not
\eqref{Eq:Conclusion_of_completeness}.  
\end{proof}

With an aim  analogous to the above but concerning their unproven stronger 
version of  Theorem~\ref{Thm:CKS}, Cramer et al.~(2002, pp.~275--276) present 
an erroneous counterexample: Their assumption ($\ast$) holds if $g$ is the 
signum function, but their conclusion ``$g=0$'' does not. 

Again in their version of  Theorem~\ref{Thm:CKS},
Cramer et al.~(2002, p.~273) formulated the completeness 
asumption~\ref{Thm:CKS}\ref{Aequiv:CKS_R_part_complete}
slightly less explicitly than we did here, 
and this apparently led to the ``clarification'' refuted as follows:
 
\begin{Counterexample}  \label{Couterex:San-M_Mouchart}  
The claim of San Martin and Mouchart~(2007, Theorem~2.1), which is the present 
Theorem~\ref{Thm:CKS} with~\ref{Aequiv:CKS_R_part_complete} replaced by
\begin{Aequiv}
\item[\rm(ii')] $\{R_{\theta_1,\theta_2}:\theta_1\in\Theta_1,\theta_2\in\Theta_2\}$ 
is complete
\end{Aequiv}
and without assumption~\ref{Aequiv:CKS_integrable},
remains false even if assumption~\ref{Aequiv:CKS_integrable} is added.  
\end{Counterexample}
\begin{proof} Let $\cX_1:=\cX_2:=\Theta_1:=\Theta_2:=\{0,1\}$, 
$\cA_i$ the power set of 
$\cX_i$, $Q_0:=R_{0,0}:=R_{0,1}:=\frac13\delta_0+\frac23\delta_1$, 
and $Q_1:=R_{1,0}:=R_{1,1}:=\frac23\delta_0+\frac13\delta_1$.
Then all the above assumptions are fulfilled, 
and with $h(x_1,x_2):=|x_1-x_2|-\frac49$, 
we have~\eqref{Eq:Hypothesis_of_completeness} but not
\eqref{Eq:Conclusion_of_completeness}.  
\end{proof}

Without any additional assumption, sufficiency of $\bigvee_{i\in I}\cC_i$
can not be added to the conclusion of Theorem~\ref{Thm:main}, 
not even in the special case of  Corollary~\ref{Cor:I={1,2}}:

\begin{Counterexample}             \label{CounterEx:C1_vee_C2_insuff}
There exists a model $\cP$ satisfying the assumptions of 
Corollary~\ref{Cor:I={1,2}}, but with $\cC_1\vee\cC_2$ insufficient. 
\end{Counterexample} 
\begin{proof} Let $\cX:=\{1,2,3\}$, $\cA:=2^\cX$, $\Theta_1:=\Theta_2:=\{1,2\}$,
and the $P_{\theta_1,\theta_2}$ defined by their densities $f_{\theta_1,\theta_2}$
with respect to counting measure given by 
\[
  f_{1,1}:=\frac13\quad\qquad f_{1,2}:= f_{2,1} :=\1_{\{3\}}\quad\qquad 
  f_{2,2}(x):=\frac{x}6 \quad\text{ for }x\in\cX.
\]
If $\theta_1=2$, then 
$\{P_{\theta_1,\theta_2}:\theta_2\in\Theta_2\}=\{P_{2,1},P_{2,2}\}$,
and for this model, by the reference given in the proof of 
Theorem~\ref{Thm:Main_hom_con}, the $\sigma$-algebra 
$\sigma(f_{2,\theta_2'}/f_{2,\theta_2''} :\theta_2',\theta_2''\in\Theta_2)
= \left\{\emptyset,\{1,2\},\{3\},\cX \right\}=:\cC_1$
is minimal sufficient and in fact  easily checked to be complete.
By similarly considering the remaining three models occurring in the 
assumptions~\ref{Cor:I={1,2}}\ref{Aequiv:Cor_theta2_fixed},\ref{Aequiv:Cor_theta1_fixed}, we see that 
the assumptions of Corollary~\ref{Cor:I={1,2}} are fulfilled
with $\cC_2:=\cC_1$. On the other hand, the  $\sigma$-algebra 
$\sigma(f_{\theta'}/f_{\theta''} :\theta',\theta''\in\Theta)=\cA$
is minimal sufficient for $\cP$, and hence 
$\cC_1\vee \cC_2$, being not almost surely equal to $\cA$,
is insufficient.
\end{proof}

\begin{Counterexample}      \label{Counterex:sigma-finiteness_needed}
Lemma~\ref{Lem:uniform_cap-stable_compl_suff}(a) would become false
if its $\sigma$-finiteness assumption were omitted, and even so if the 
assumption of~\ref{Lem:uniform_cap-stable_compl_suff}(b) were added.
\end{Counterexample}
\begin{proof}
Let $(\cX,\cA,\mu)=(\R,\cB(\R),\#$) be the real line with its 
Borel $\sigma$-algebra and counting measure, 
$\cE$ consist of all singletons and the empty set, 
and $n=1$. Then  $\cE$ is $\cap$-stable.
If $A\in\cA$, then, for every $x\in\cX$,
$h\in \delta_x(A\pmb{|}\cC)$ implies $h(x)=\1_A(x)$, and hence
$h\in\bigcap_{P\in\cP}P(A\pmb{|}\cC)$ would imply $h=\1_A$, but the latter
is not $\cC$-measurable if $A$ is neither countable nor co-countable.
So $\cC$ is not sufficient. 
\end{proof}

\section*{Acknowledgements}
We thank Todor Dinev and Christoph Tasto for their help with the proofreading.

\end{document}